%% file: Henon-Kato_v6_arxiv.tex
\documentclass[12pt,centertags,oneside]{amsart}

\usepackage{amsmath,amstext,amsthm,amscd,typearea,hyperref,stmaryrd,mathabx}
\usepackage{amssymb}
\usepackage{a4wide}
\usepackage[mathscr]{eucal}
\usepackage{mathrsfs}
\usepackage{typearea}
\usepackage{charter}
\usepackage{pdfsync}
\usepackage[a4paper,width=16.2cm,top=3cm,bottom=3cm,lines=50]{geometry}
\usepackage[all]{xy}
\usepackage{tikz}
\usepackage{enumitem}
\usepackage[foot]{amsaddr}
\usepackage{academicons}
\usepackage{xcolor}

\numberwithin{equation}{section}
\newcommand{\orcid}[1]{\href{https://orcid.org/#1}{\textsc{orc}i\textsc{d}}}

\title[]{H\'enon maps with biholomorphic Kato surfaces}

\author{Fran\c cois Bacher}
\address{Université Bourgogne Europe, CNRS, IMB UMR 5584, F-21000 Dijon, France}
\email{francois.bacher@ube.fr}
\date{\today}
\newcommand{\bigslant}[2]{{\raisebox{.4em}{$#1$}\left/\raisebox{-.4em}{$#2$}\right.}}

\theoremstyle{plain}
\newtheorem{thm}{Theorem}[section]
\newtheorem{lem}[thm]{Lemma}

\newtheorem{prop}[thm]{Proposition}

\newtheorem*{thm*}{Theorem}
\newtheorem*{conj*}{Conjecture}

\theoremstyle{definition}
\newtheorem{defn}[thm]{Definition}

\newtheorem*{exmp*}{Example}

\theoremstyle{remark}

\makeatletter

\@addtoreset{equation}{section}
\makeatother

\DeclareMathOperator{\Aut}{Aut}

\newcommand{\eps}{\varepsilon}

\newcommand{\set}[1]{\mathbb{#1}}

\newcommand{\germ}[2]{\left(#1,#2\right)}

\newcommand{\der}[2]{\frac{\partial#1}{\partial#2}}

\newcommand{\norm}[1]{\left\Vert#1\right\Vert}

\newcommand{\intoo}[2]{\left(#1,#2\right)}

\newcommand{\intent}[2]{\left\llbracket#1,#2\right\rrbracket}

\newcommand{\proj}[1]{\set{P}^{#1}}

\newcommand{\setst}{~\vert~}

\newcommand{\wo}[2]{#1\backslash#2}

\newcommand{\wt}[1]{\widetilde{#1}}

\newcommand{\wh}[1]{\widehat{#1}}

\begin{document}

\theoremstyle{plain}

\begin{abstract} Let $F$ and~$G$ be generalized H\'enon maps. We show that the Kato surfaces associated to the germs of~$F$ and~$G$ near infinity are biholomorphic if and only if~$F$ and~$G$ are conjugate in $\Aut\left(\set{C}^2\right)$. This answers a question raised by Favre in the survey~\cite{Xen}. Moreover, we describe all the Kato surfaces associated to the germ of a generalized H\'enon map near infinity.

  We also show that two generalized H\'enon maps are conjugate near infinity, if and only if they are affinely conjugate.
  \end{abstract}

\maketitle

\section{Introduction}

This paper investigates the relationship between two different types of objects: on the one hand, the iconic family of dynamical systems known as generalized Hénon maps; on the other hand, a class of non-Kähler complex surfaces known as Kato surfaces. Our main result, in rough terms, establishes an injection from the moduli space associated with the former into that of the latter, along with a detailed description of its image.

In 1976, Hénon~\cite{Henon} introduced the family of maps
\[H_{a,c}\colon\set{R}^2\to\set{R}^2;\qquad (x,y)\mapsto \left(ay+x^2+c,ax\right),\]
for $a\in\set{R}^*$ and $c\in\set{R}$, as a simplified model for the Poincaré first return map of the Lorenz flow. He carried out numerical experiments for some parameters and exhibited that they seemed to have strange attractors, with a fractal structure. H\'enon's observation were proven to be true by Benedicks and Carleson~\cite{BeneCarl} in 1991, using the works of \mbox{Jakobson}~\cite{Jakob} on $1$-dimensional maps. Moreover, H\'enon maps provide a natural framework for studying homoclinic tangencies, a central mechanism for bifurcations. They also play a universal role in this context: in many situations, renormalization near the unfolding of a homoclinic tangency gives rise to H\'enon-like dynamics.
%This pioneer work was further pursued by many authors, namely Mora--Viana~\cite{MoraViana}, Wang--Young~\cite{WangYoung}, Takahashi~\cite{Taka} and Berger~\cite{Berger}

Here, we are interested in a generalization of the maps Hénon introduced in the complex framework. The study of complex H\'enon maps was actually developed parallelly, first with algebraic me\-thods. It was known since Jung~\cite{Jung} that the group of polynomial automorphisms of~$\set{C}^2$ has the structure of an amalgamated product of affine maps and elementary maps, i.e., maps which preserve the pencil of lines $\{x=\rm{constant}\}$. In 1989, Friedland and \mbox{Milnor}~\cite{FriMil} used Jung's Theorem to show that polynomial automorphisms of~$\set{C}^2$ are either conjugated to an affine or to an elementary map, or to a so-called generalized H\'enon map. Among these three types, the last one is the only dynamically non-trivial, that is, the maps of which have positive entropy. In this paper, we will call \emph{generalized H\'enon maps} automorphisms $F=F_N\circ\dots\circ F_1$ of $\set{C}^2$, for $N\in\set{N}^*$ and $F_i$, $i\in\{1,\dots,N\}$, with
\begin{equation}\label{defFiintro}F_i\colon\set{C}^2\to\set{C}^2,\qquad(z,w)\mapsto\left(P_i(z)-a_iw,z\right),\end{equation}
for~$P_i$ a polynomial of degree at least~$2$ and~$a_i\in\set{C}^*$. Maps of the form~\eqref{defFiintro}, will be called \emph{standard H\'enon maps}. Up to an affine conjugacy, it is easy to realize $H_{a,c}$ as the restriction to~$\set{R}^2$ of a standard H\'enon map in the latter sense.

In the early 90's, the fruitful analytic approach to $1$-dimensional dynamics was deve\-loped in higher dimension. Arguably, the easiest case in this direction is the study of automorphisms of~$\set{C}^2$. This method was carried out simultaneously by Hubbard--Oberste-Vorth~\cite{HOV}, Forn\ae{}ss--Sibony~\cite{ForSibHenon} and Bedford--Smillie, in a long series of articles, see for example~\cite{BedSmiI,BedSmiIII,BedSmiVI}. Part of this series~\cite{BedSmiIV} was written with Lyubich and has also shown to be fruitful in the study of real H\'enon maps. Since then, H\'enon maps have never ceased to be a central object in holomorphic dynamics, with active research, as can be seen in the survey of open questions~\cite{Xen}.%More precisely, their methods generally involve constructing positive closed currents, studying their laminarity and intersection properties.

On the other hand, this paper deals with Kato surfaces. These surfaces play an important role in Kodaira's classification of compact minimal complex surfaces~\cite{Kodclass}. In Kodaira's~\cite{Kod66} presentation, they belong to the class $\mathrm{VII}_0$. These are the surfaces with Kodaira dimension $-\infty$ and first Betti number~$1$. In particular, these are non-K\"{a}hler surfaces. The class $\mathrm{VII}_0$ is the last mysterious gap in the Kodaira classification. The simplest examples in this class are primary Hopf surfaces, given as quotients of $\wo{\set{C}^2}{\{0\}}$ by the infinite cyclic group generated by the linear automorphism $(z_1,z_2)\mapsto(\alpha_1z_1,\alpha_2z_2)$, $\alpha_1,\alpha_2\in\set{D}^*$. They admit a \emph{global spherical shell}, that is, a holomorphic embedding of a neighbourhood of the sphere $\set{S}^3\subset\set{C}^2$ that does not disconnect the surface.

More generally, consider a tower of blow-ups $\Pi\colon\wh{B}\to\germ{\set{C}^2}{0}$ and a germ of biholomorphism $\sigma\colon(\set{C}^2,0)\to(\wh{B},p)$, for some $p\in\wh{B}$. For two well chosen balls $B\subset B'\subset\set{C}^2$ centered at~$0$, consider $\wo{\Pi^{-1}(B')}{\sigma(B)}$ and glue it along $\sigma\circ\Pi$ (see Section~\ref{secKato1} and Figure~\ref{figconstructionKato} for more details). The obtained surface admits a global spherical shell and belongs to the $\mathrm{VII}_0$ class. This construction was initiated by Kato~\cite{Kato}, later studied by Enoki~\cite{Enoki1,Enoki2} and Nakamura~\cite{Nak1,Nak2} and extensively by Dloussky with others~\cite{Dlou1,Dlou2,Dlou3,Dlou4} (see the memoir~\cite{memDlou}). Many of these works manage to translate interesting geometric pro\-perties into combinatorial properties of the successive blow-ups, making them relatively elementary to study. Among surfaces of class $\mathrm{VII}_0$, the surfaces obtained by this procedure are the only one known, and are conjecturally the only one existing. Moreover, it was shown by Dloussky~\cite{memDlou} that minimal surfaces admitting a global spherical shell can be constructed that way. So, one can reformulate the conjecture to be that every surface of class $\mathrm{VII}_0$ would admit a global spherical shell. Proving such a conjecture would close Kodaira's classification.

In the present paper, we are interested in a bridge between these two objects of H\'enon maps and Kato surfaces. Consider a generalized H\'enon map $F=F_N\circ\dots\circ F_1$, for $N\in\set{N}^*$ and $F_i$ of the form~\eqref{defFiintro}, $i\in\{1,\dots,N\}$. The maps~$F$ and $F^{-1}$ extend to birational maps $\proj{2}\dashrightarrow\proj{2}$ with respective indeterminacy points~$I^{\pm}$. In particular, they are regular in the sense of Sibony~\cite{SibPanSyn}, i.e., $I^+$ and $I^-$ are disjoint. In homogeneous coordinates $[t:z:w]$, the extension of~$F$ (resp.~$F^{-1}$) contracts the whole line at infinity $\{t=0\}$ onto the super-attractive fixed point~$I^-=[0:1:0]$ (resp. $I^+=[0:0:1]$). As for one-dimensional polynomial dynamics, it is useful to consider the dichotomy between points with escaping and points with bounded dynamics. Indeed, if we denote by
\[\Omega^{\pm}=\left\{p\in\set{C}^2\setst F^{\pm n}(p)\to_{n\to\infty}I^{\mp}\right\}=\left\{p\in\set{C}^2\setst F^{\pm n}(p)\text{ is not bounded}\right\},\]
by $J^{\pm}=\partial\Omega^{\pm}$ and $J=J^+\cap J^-$, then~$J$ is often a very good analog of the Julia set of a rational map of~$\proj{1}$. Recall that for a polynomial $f\colon\set{C}\to\set{C}$, the set $K_f$ of points with bounded orbits is the filled-in Julia set, and $J_f=\partial K_f$ is the Julia set of~$f$. This is where is concentrated all the chaotic behaviour of~$f$, and this is the support of the unique measure of maximal entropy. In their series, Bedford and Smillie showed analogous results for~$J$ in the case of H\'enon maps. If we concentrate on~$\Omega^+$, Hubbard and Oberste-Vorth~\cite{HOV} gave a full analytic description of this basin of attraction, and in particular showed that~$F$ acts properly and discontinuously on~$\Omega^+$. Therefore, one can define the quotient $\Omega^+/F$, being naturally a complex surface.

As Dloussky and Oeljeklaus~\cite{HenonDlou} and Favre~\cite{Fav} noticed, this quotient can be compactified by the above procedure into a Kato surface. One obtains an intermediate Kato surface, admitting a unique holomorphic foliation~\cite{DlOefol}, which is given by the level sets of a Green function. %If the Jacobian of the H\'enon map equals its degree, the foliation is even given by a global vector field on the Kato surface.
In fact, this foliation comes from one on~$\Omega^+$, which plays a crucial role in the works of Hubbard--Oberste-Vorth~\cite{HOV} and Bedford--Smillie~\cite{BedSmiV,BedSmiVI}.
%This foliation, when considered one the basin of attraction $\Omega^+$, has proven to be fruitful to study by the works of Hubbard--Oberste-Vorth~\cite{HOV} and Bedford--Smillie~\cite{BedSmiV,BedSmiVI}.

As a more general bridge between dynamics and Kato surfaces, the classification of such surfaces is deeply linked to the one of \emph{rigid germs} in the sense of Favre~\cite{Fav}. These are the holomorphic germs $f\colon\germ{\set{C}^2}{0}\to\germ{\set{C}^2}{0}$ such that if we denote by $\mathcal{C}_f$ the critical set of~$f$, the set $\mathcal{C}^{\infty}=\cup_{n\in\set{N}}f^{-n}(\mathcal{C}_f)$ forms a normal crossing divisor. In~\cite{Fav} and in his thesis~\cite{theseFav}, Favre gave a complete classification of such germs and determined which ones give rise to Kato surfaces. Also, he gave simple normal forms which proved to be very useful in the recent study of Kato surfaces (see for example~\cite{OelToma}). In the recent survey~\cite{Xen}, he asked the following question. Given a generalized H\'enon map~$F$, can one describe all the generalized H\'enon maps~$G$ such that $S_F$ is biholomorphic to~$S_G$? It is a straightforward consequence of~\cite{Fav} and~\cite{HOV} that if two generalized H\'enon maps are conjugate near infinity (that is, near~$I^-$), then they have biholomorphic escaping set. It is also well known~\cite{memDlou} that two conjugate germs induce the same Kato surface. Therefore, the works of Bonnot--Radu--Tanase~\cite{BRT} and Pal~\cite{Pal} can be seen as partial results towards a solution to this question. They investigated when two standard H\'enon maps have biholomorphic escaping sets. However, their condition is obviously weaker than the one to be found for Kato surfaces. Indeed, on the one hand they find H\'enon maps having different Jacobian with biholomorphic escaping sets, and on the other hand Favre~\cite{Fav} showed that the germ at infinity of a H\'enon map determines the Jacobian of the map. Here, we prove the following.

\begin{thm} \label{critKato} Let $F=F_N\circ\dots\circ F_1$ and~$G$ be generalized H\'enon maps, where the $F_i$, $i\in\{1,\dots,N\}$ are standard H\'enon maps. Then, $S_F$ and $S_G$ are biholomorphic if and only if~$G$ is affinely conjugate to one of the $F_k\circ\dots\circ F_1\circ F_N\circ\dots\circ F_{k+1}$, for $k\in\{1,\dots,N\}$.
\end{thm}

This condition is equivalent to~$F$ and~$G$ being conjugate in $\Aut\left(\set{C}^2\right)$~\cite{FriMil}. In terms of moduli space, this means that we have an injection from the moduli space of H\'enon maps onto the moduli space of Kato surfaces. As we will show, we also have some kind of surjection property on the moduli space of Kato surfaces of a certain type. More precisely, one obtains a surjection if the surface does not admit a global vector field, and we describe the image if it does. In this case, the germs at infinity of H\'enon maps give new normal forms for germs inducing Kato surfaces.

Let us describe briefly the method of our proof. First, we show that $S_F\simeq S_G$ if and only if $G$ is conjugate near infinity to one of the $F_k\circ\dots\circ F_1\circ F_N\circ\dots\circ F_{k+1}$. This is actually a consequence of some well known properties of Kato surfaces and a basic analysis of the Kato surface induced by a generalized H\'enon map. Indeed, Dloussky~\cite{memDlou} showed that a Kato surface $S$ contains~$n$ compact curves, coming with $n$ contracting germs $f_1,\dots,f_n$, for $n=b_2(S)$. These curves are somehow cyclically ordered in the sense that one curve is obtained as the exceptional divisor of the blowing-up in one point of the previous one. The cyclic behaviour is then a consequence of taking the quotient. Dloussky~\cite{memDlou} proved that two Kato surfaces inducing respectively germs $f_1,\dots,f_n$, $g_1,\dots,g_n$ are biholomorphic if and only if one of the $f_i$ is conjugate to one of the $g_j$. In the case of a composition $F_N\circ\dots\circ F_1$, where each map $F_j$ induces separately a Kato surface, we show that each $F_k\circ\dots\circ F_1\circ F_N\circ\dots\circ F_{k+1}$ is conjugate to one of the $f_j$.

So, the first part of the proof reduces to showing that these are the only ones which can be conjugate to a generalized H\'enon map near infinity. We prove the latter by computing the configuration of self-intersection of the~$n$ curves mentioned above, in which there are exactly~$N$ curves which have particular self-intersection and correspond to the germs of the generalized H\'enon maps $F_k\circ\dots\circ F_1\circ F_N\circ\dots\circ F_{k+1}$.

Next, we show that two generalized H\'enon maps are conjugate near infinity if and only if they are affinely conjugate. This is the most involving part of our work. We do so by computing the normal form of the germ in the sense of Favre~\cite{Fav}. Such a computation was done up to degree~$6$ for standard H\'enon maps by Bonnot--Radu--Tanase~\cite{BRT}. Favre gave an explicit procedure in three steps to conjugate a rigid germ to a normal form \emph{via} an (essentially) tangent to the identity diffeomorphism. We follow this procedure and keep track of enough (exactly $2d-1$, where~$d$ is the degree of~$F$) of the first terms of the intermediate maps. More precisely, the first step of this procedure consists of taking B\"{o}ttcher coordinates, and we show that the next two steps more or less consist of cutting all the terms of degree more than~$2d-1$.

Having computed this normal form, we show that this procedure is injective if the right initial choices were made, namely taking monic and centered polynomials. That is, if two generalized H\'enon maps induce the same normal form, then they are equal (not only conjugate). In some sense, we show that these $2d-1$ terms contain all the polynomials and Jacobians of the standard H\'enon maps in such a way that they do not ``interfere''. That is, the first terms are only determined by the polynomial~$P_N$ (see~\eqref{defFiintro}) and uniquely determines it. The other polynomials only intervene after inside the full development. This property is also satisfied inductively and we recover the whole generalized H\'enon map from its germ. Next, we use the work of Favre~\cite{Fav}, who gives necessary and sufficient conditions for two normal forms to be conjugate. We show in our case, that conjugate normal forms are given by affinely conjugate generalized H\'enon maps.

Finally, we show that the map taking a H\'enon map to a Kato surface of certain self-intersection profile is essentially surjective (except when the Jacobian equals the degree, see Theorem~\ref{surjHenonKato}). Combined with the preceeding results, this shows that generalized H\'enon maps with monic and centered polynomials give normal forms for germs inducing this type of Kato surfaces. To do so, we prove that the natural coefficients of the H\'enon map $(\alpha_i)_{i\in\{1,\dots,K\}}$, and those of the polynomial normal form $(\beta_i)_{i\in\{1,\dots,K\}}$, are of the same amount~$K$, and if put in the right order, depend triangularly and algebraically from one to another. That is, 
\[\beta_i=u_i\alpha_i+P_i(\alpha_1,\dots,\alpha_{i-1}),\]
for $u_i\in\set{C}^*$ and $P_i$ universal polynomials. Such a map $(\alpha_i)_{i\in\{1,\dots,K\}}\mapsto(\beta_i)_{i\in\{1,\dots,K\}}$ is easily invertible. Here, we use extensively the description of intermediate Kato surfaces of a given self-intersection profile given by Oeljeklaus and Toma~\cite{OelToma}.

The article is organized as follows. In Section~\ref{secKato1}, we recall some basic constructions of Kato surfaces. We show that if~$f$ and~$g$ induce Kato surfaces, then $S_{f\circ g}\simeq S_{g\circ f}$. In Section~\ref{secKato2}, we compute the self-intersection configuration of Kato surfaces induced by generalized H\'enon maps. We prove our first criterion involving germs at infinity. Next, we compute the normal form of a germ at infinity of a generalized H\'enon map in Section~\ref{secnormalform} and prove Theorem~\ref{critKato}. Finally, we conclude with showing the surjectivity in Section~\ref{secsurj}.

\subsection*{Notations} Throughout the article, we denote by $\intent{p}{q}$ the set of integers $p\leq k\leq q$.

\subsection*{Acknowledgments}  This work has been supported by the EIPHI Graduate school (contract ``ANR-17-EURE-0002'') and by the Région ``Bourgogne Franche-Comté''. The author would like to thank the referee for his interesting remarks and questions.

% The author would like to thank Johan Taflin, for pointing out to him this interesting problem.
\section{Kato surfaces}\label{secKato1}

In this section, we recall some basic constructions of Kato surfaces. We refer the reader to Dloussky's memoir~\cite{memDlou} or to Favre's thesis~\cite{theseFav} for more details and proofs.

\begin{defn} A \emph{tower of blow-ups} is a map $\Pi=\Pi_1\circ\dots\circ\Pi_n$, where $\Pi_1\colon B_1\to\germ{\set{C}^2}{0}$ is a blow-up of the origin $O_0=0\in\set{C}^2$, and for each $i\in\intent{2}{n}$, $\Pi_i\colon B_i\to B_{i-1}$ is the blow-up of some $O_{i-1}\in\Pi_{i-1}^{-1}(O_{i-2})$. We usually denote $\wh{B}=B_n$ and $B_0=\set{C}^2$.

  A contracting germ $F\colon\germ{\set{C}^2}{0}\to\germ{\set{C}^2}{0}$ is called a \emph{Dloussky germ} if it can be written as $F=\Pi\circ\sigma$, where $\Pi=\Pi_1\circ\dots\circ\Pi_n\colon\wh{B}\to\germ{\set{C}^2}{0}$ is a tower of blow-ups and with the notations above, $\sigma\colon\germ{\set{C}^2}{0}\to\germ{\wh{B}}{O_n}$ is a germ of biholomorphism, for some point $O_n\in\Pi_n^{-1}(O_{n-1})$.
\end{defn}

Let $F=\Pi\circ\sigma$ be a Dloussky germ, defined on a ball $B\subset\set{C}^2$ of radius $2r>0$. Let $\eps\in\intoo{0}{r}$ be sufficiently small and denote by $\Sigma_{\eps}=\{z\in\set{C}^2\setst\norm{z}\in\intoo{r-\eps}{r+\eps}\}$. The map $\sigma\circ\Pi\colon\wh{B}\to\wh{B}$ sends biholomorphically $\Pi^{-1}\left(\Sigma_{\eps}\right)$ onto $\sigma\left(\Sigma_{\eps}\right)$. Denote by $B_{\eps}$ the ball of radius $r+\eps$ and $B_{-\eps}$ the ball of radius $r-\eps$. One can define the \emph{Kato surface associated to~$F$} to be
\[S_F=\bigslant{\left(\Pi^{-1}\left(B_{\eps}\right)\setminus\sigma\left(B_{-\eps}\right)\right)}{\sim},\]
where~$\sim$ is the gluing given by $\sigma\circ\Pi\colon\Pi^{-1}\left(\Sigma_{\eps}\right)\to\sigma\left(\Sigma_{\eps}\right)$, that is $z\sim\sigma\circ\Pi(z)$, for $z\in\Pi^{-1}\left(\Sigma_{\eps}\right)$ (see Figure~\ref{figconstructionKato}). One can show that this construction only depends on the germ~$F$ and not on the other choices made~\cite[Proposition~3.16]{memDlou}. We have the following.

\begin{figure}[htb]
  \centering
  \input{Construction-Kato.tex}
  \caption{\label{figconstructionKato}Construction of Kato surfaces}
\end{figure}
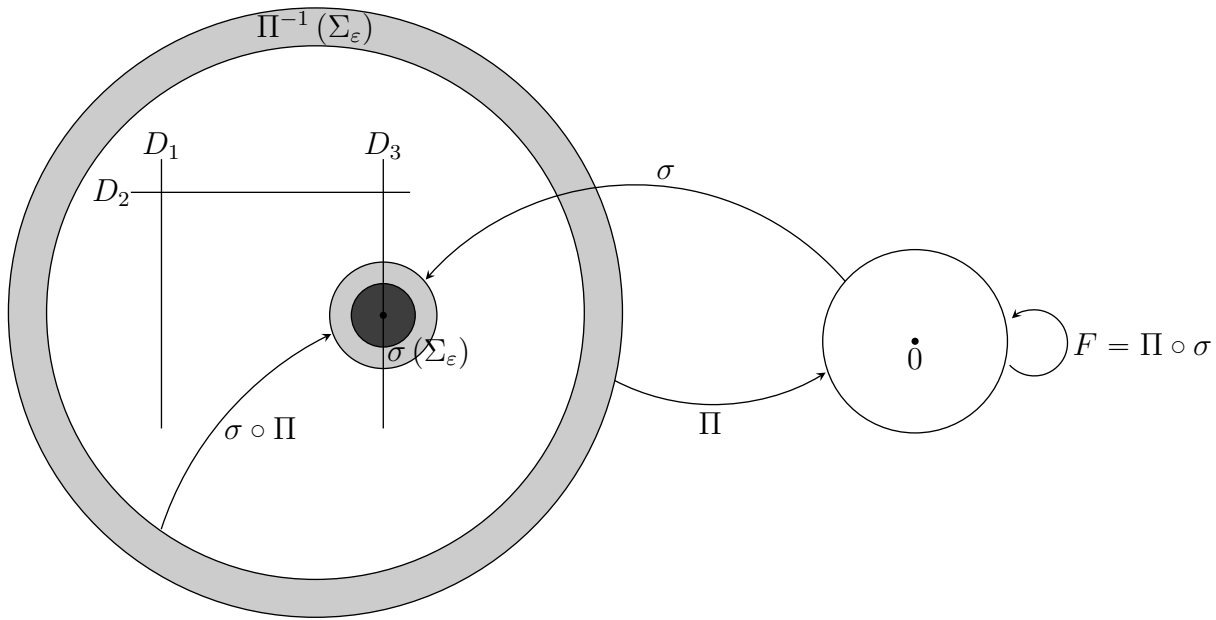

\begin{prop}[Dloussky~{\cite[Lemma~1.10]{memDlou}}] With the notations above, $b_1(S_F)=1$ and $b_2(S_F)=n$ is the number of blow-ups. In particular, $S_F$ is non-K\"{a}hler.
\end{prop}

Moreover, the Kato surface $S_F$ comes with~$n$ compact curves $D_1,\dots,D_n$, which are the respective projections of the exceptional divisor of the respective blow-ups $\Pi_1,\dots,\Pi_n$. Each of this curve comes with a contracting germ
\begin{equation}\label{defFk}F_k=\Pi_{k+1}\circ\dots\circ\Pi_n\circ\sigma\circ\Pi_1\circ\dots\circ\Pi_k\colon\germ{B_k}{O_k}\to\germ{B_k}{O_k},\qquad k\in\intent{1}{n}.\end{equation}
By definition, $F_n=\sigma\circ F\circ\sigma^{-1}$. Moreover, we have the following description of Dloussky germs inducing the same Kato surface.

\begin{thm}[Dloussky~{\cite[Subsection~3.11]{memDlou}}]\label{thmDlou} Let~$F$ be as above and~$G$ be a Dloussky germ. Then, $S_G$ is biholomorphic to $S_F$, if and only if~$G$ is conjugate to one of the~$F_k$, $k\in\intent{1}{n}$.
\end{thm}

In particular, we have the following. The first part of Proposition~2.4 is due to Favre in his thesis~\cite[Lemma~1.2.11]{theseFav}, but we reprove it to introduce notations.

\begin{prop}\label{Sfg=Sgf} Let $F,G$ be two Dloussky germs. Then, $G\circ F$ and $F\circ G$ are Dloussky germs and $S_{F\circ G}$ is biholomorphic to $S_{G\circ F}$.
\end{prop}

\begin{proof} Let us write $F=\Pi_1\circ\sigma_1$ and $G=\Pi_2\circ\sigma_2$, with $\Pi_j\colon B_j\to\set{C}^2$, $j\in\{1,2\}$ towers of blow-ups. Denote by $O_1=\sigma_1(0)$. Because blowing-up does not depend on the coordinates, we have that~$\sigma_1$ lifts to a blow-up of~$O_1$. That is, there is $\wt{\Pi}_2\colon\wt{B}_2\to B_1$ a tower of blow-ups of~$O_1$ and $\wt{\sigma}_1\colon B_2\to\wt{B}_2$ such that $\wt{\Pi}_2\circ\wt{\sigma}_1=\sigma_1\circ\Pi_2$. Let us summarize our data in a commutative diagram.
  \[\xymatrix{ & & \wt{B}_2\ar[d]^{\wt{\Pi}_2} \\ B_1\ar[d]_{\Pi_1}& B_2\ar[ur]^{\wt{\sigma}_1}\ar[d]_{\Pi_2} & B_1\ar[d]^{\Pi_1} \\ \germ{\set{C}^2}{0} \ar[ur]^{\sigma_2} \ar[r]_{G} & \germ{\set{C}^2}{0} \ar[ur]^{\sigma_1} \ar[r]_{F} & \germ{\set{C}^2}{0}}\]
  This already shows that $F\circ G$ is a Dloussky germ, and the proof is of course similar for $G\circ F$. Moreover, by Theorem~\ref{thmDlou}, the germ $\wt{\Pi}_2\circ\wt{\sigma}_1\circ\sigma_2\circ\Pi_1$ and its conjugate induce the same Kato surface as $F\circ G$. Finally, it is easily checked that
  \[\sigma_1^{-1}\circ\wt{\Pi}_2\circ\wt{\sigma}_1\circ\sigma_2\circ\Pi_1\circ\sigma_1=G\circ F.\qedhere\]
\end{proof}

Let us now conclude this section with some important remarks about the~$n$ curves $D_1,\dots,D_n$ mentioned above. Consider $G$ a Dloussky germ such that $b_2(S_G)=b_2(S_F)$. Write $G=\Pi'\circ\sigma'$, where $\Pi'=\Pi'_1\circ\dots\circ\Pi'_n$ is a tower of blow-ups. Denote by $E_1,\dots,E_n$ the~$n$ curves in $S_G$ given by the exceptional divisors of $\Pi'_1,\dots,\Pi'_n$, and by $G_k$ the germs defined similarly to~\eqref{defFk}.

\begin{prop}[Dloussky~{\cite[Remark~3.15]{memDlou}}]\label{biholoKatocourbes} Suppose there is $\varphi\colon S_F\to S_G$ a biholomorphism. Then, there is a $k\in\intent{1}{n}$ such that $\varphi(D_j)=E_{j+k}$, $F_j$ is conjugate to $G_{j+k}$, for $j\in\intent{1}{n-k}$; $\varphi(D_j)=E_{j+k-n}$ and $F_j$ is conjugate to $G_{j+k-n}$, for $j\in\intent{n-k+1}{n}$.
\end{prop}

\section{Self-intersection profile of a H\'enon map's Kato surface} \label{secKato2}

In this section, we compute the sequence of blow-ups and the biholomorphism that make the germ of a H\'enon map near~$I^-$ a Dloussky germ. The result is essentially contained in~\cite[Proposition~2.2]{Fav} and~\cite[pp.~25--27]{theseFav} for standard H\'enon maps and the first part of Proposition~\ref{Sfg=Sgf}, due to Favre, gives it for generalized H\'enon maps. It was also already observed by Dloussky and Oeljeklaus~\cite{HenonDlou} for degree~$2$ H\'enon maps. Therefore, our goal is not really this computation but the following.

\begin{prop}\label{critKatoloc} Let $F=F_N\circ\dots\circ F_1$ and $G$ be two generalized H\'enon maps, where $F_1,\dots,F_N$ are standard H\'enon maps. Then, the Kato surfaces $S_F$ and $S_G$ are biholomorphic if and only if $G$ is conjugate near infinity to one of the $F_k\circ\dots\circ F_1\circ F_N\circ\dots\circ F_{k+1}$, $k\in\intent{1}{N}$.
\end{prop}

Here, we say that two generalized H\'enon maps $F$ and $G$ are \emph{conjugate near infinity} if their germs near the indeterminacy point $I^-$ of $F^{-1}$ and $G^{-1}$ are conjugate (as germs).

Let us first consider a standard H\'enon map
\[F\colon\set{C}^2\to\set{C}^2,\quad(z,w)\mapsto(P(z)-aw,z),\]
for some polynomial $P$ of degree $d\geq2$ and $a\in\set{C}^*$. In affine coordinates centered at $I^-$, it can be written as
\[F(t,w)=\left(\frac{t^d}{U(t)-awt^{d-1}},\frac{t^{d-1}}{U(t)-awt^{d-1}}\right),\]
where $U(t)=t^dP\left(\frac{1}{t}\right)$. So, if $\Pi_1(u,v)=(uv,v)$, $\Pi_2(u,v)=(u,u^{d-1}v)$, one can write $F=\Pi_1\circ\Pi_2\circ\eta$, for
\[\eta(t,w)=\left(t,\frac{1}{U(t)-awt^{d-1}}\right).\]
Now, consider $\Pi_{\gamma}(u,v)=(u,uv+\gamma)$, for $\gamma\in\set{C}$ and write $\frac{1}{U(t)}=\sum_{j=0}^{d-2}a_jt^j+t^{d-1}B(t)$, for~$B$ a holomorphic map. Note that $a_0\neq0$, since $P$ is a polynomial of degree~$d$. It is easily checked that
\[\eta=\Pi_{a_0}\circ\dots\circ\Pi_{a_{d-2}}\circ\sigma,\]
for
\begin{equation}\label{sigmaKatoHenon}\sigma(t,w)=\left(t,B(t)+\frac{aw}{U(t)\left(U(t)-awt^{d-1}\right)}\right).\end{equation}
It is clear that $\sigma$ is a germ of biholomorphism, and since $\Pi_1$, $\Pi_2$ and the $\Pi_{a_j}$ are blow-ups in standard charts, this makes $F=\Pi_1\circ\Pi_2\circ\Pi_3\circ\sigma$, for $\Pi_3=\Pi_{a_0}\circ\dots\circ\Pi_{a_{d-2}}$, a Dloussky germ. Combined with Proposition~\ref{Sfg=Sgf}, we obtain the following.

\begin{lem}\label{HenonDlousskyb2} Let $F=F_N\circ\dots\circ F_1$ be a generalized H\'enon map, where the $F_i$ are standard H\'enon maps of degree $d_i$, $i\in\intent{1}{N}$. Then, the germ at infinity of~$F$ is a Dloussky germ inducing a Kato surface $S_F$ with $b_2(S_F)=\sum_{k=1}^N\left(2d_k-1\right)$. 
\end{lem}

Denote by $n=b_2\left(S_F\right)$ and by $D_1,\dots,D_n$ the~$n$ compact curves which are the projections of the exceptional divisors in $S_F$. For $i\in\intent{1}{n}$, denote by $D_i\cdot D_i$ the number of self-intersections of~$D_i$.

\begin{lem}\label{selfinterprofil} With the notations of Lemma~\ref{HenonDlousskyb2}, the self-intersection profile of $S_F$, that is the $n$-tuple $(-D_1\cdot D_1,\dots,-D_n\cdot D_n)$ is
  \begin{equation}\label{DlSHenongen}(d_1,\underbrace{2,\dots,2}_{2d_1-3~\text{times}},3,d_2,\underbrace{2,\dots,2}_{2d_2-3~\text{times}},3,\dots,d_N\underbrace{2,\dots,2}_{2d_N-3~\text{times}},3).\end{equation}
\end{lem}

\begin{proof} This is essentially counting and following strict transforms of exceptional divisors. Consider first a standard H\'enon map and keep notations above Lemma~\ref{HenonDlousskyb2}. The map~$\Pi_2$ consists of blowing up $d-1$ times a point of the strict transform of~$D_1$. Indeed, the strict transform of $D_1$ is defined in coordinates by $\{v=0\}$. Next, $\Pi_3$ blows up a point which belongs to $D_{d}$, but not to any other divisor since $a_0\neq0$. All the other maps of~$\Pi_3$ blow up a point that is only in the last appeared exceptional divisor. Finally, $D_n=\{u=0\}$ in the last coordinates, and by~\eqref{sigmaKatoHenon} is glued with $\{u=0\}$ via $\sigma\circ\Pi$ in the Kato surface. Therefore, it will be blown up by $\Pi_1$, by the first blow-up of $\Pi_2$ and will be left alone after. We obtain the result for standard H\'enon maps.

  We let the reader check the result for generalized H\'enon maps, using Proposition~\ref{Sfg=Sgf} and the form of the map~$\sigma$ in~\eqref{sigmaKatoHenon} to follow the last exceptional divisor of each standard H\'enon map.
\end{proof}

Before concluding the section with the proof of Proposition~\ref{critKatoloc}, let us mention the following consequence of the above computation, that we will use in the next section.
  
\begin{lem}\label{Henonlocegaldeg} Let $F=F_N\circ\dots\circ F_1$ and $G=G_P\circ\dots\circ G_1$ be generalized H\'enon maps, with $F_N,\dots,F_1,G_P,\dots,G_1$ standard H\'enon maps. Suppose that~$F$ and~$G$ are conjugate near infinity. Then, $N=P$ and $\deg F_k=\deg G_k$, for each $k\in\intent{1}{N}$.
\end{lem}

\begin{proof}[Proof of Proposition~\ref{critKatoloc}] Suppose that~$G$ is conjugate to $F_k\circ\dots\circ F_1\circ F_N\circ\dots\circ F_{k+1}$ near infinity. Then, Proposition~\ref{Sfg=Sgf} clearly implies that $S_F$ is biholomorphic to $S_G$.

  Conversely, suppose that $S_F$ and $S_G$ are biholomorphic. We apply Proposition~\ref{biholoKatocourbes} and use its notations. Then, we should have $D_j\cdot D_j=E_{j+k}\cdot E_{j+k}$, for $j\in\intent{1}{n-k}$ and $D_j\cdot D_j=E_{j+k-n}\cdot E_{j+k-n}$, for $j\in\intent{n-k+1}{n}$. Now, we use Lemma~\ref{selfinterprofil}. Note that the~$3$ which appear in our writing (that is, not the~$d_i$ which could also be equal to~$3$) are the only ones which are (cyclically) preceeded by a~$2$. Therefore, they are sent one to another by~$\varphi$. Since $G_n$ is conjugate to~$G$, we obtain the result by applying once again Proposition~\ref{biholoKatocourbes}.
\end{proof}

\section{The germ at infinity of H\'enon maps}\label{secnormalform}

In this section, we prove the following. Combined with Proposition~\ref{critKatoloc}, we obtain Theorem~\ref{critKato}.

\begin{thm}\label{Henonlocaff} Two generalized H\'enon maps are conjugate near infinity if and only if they are affinely conjugate.
\end{thm}

Our proof involves computing the normal form of a generalized H\'enon map's germ near infinity in the sense of Favre~\cite{Fav}. We follow the three steps of his proof to obtain one particular normal form and then use his criterion to see which ones are conjugate.

First, let us introduce some notations. Let $F=F_N\circ\dots\circ F_1$ be a generalized H\'enon map, where the $F_i$, $i\in\intent{1}{N}$ are standard H\'enon maps. That is, $F_i(z,w)=\left(P_i(z)-a_iw,z\right)$, $(z,w)\in\set{C}^2$. Denote by $d_i=\deg P_i$ and $D_i=\prod_{j=1}^id_j$, $i\in\intent{1}{N}$. Note that $D_N=\deg F$. By convention, set also $D_0=1$. We will suppose that the polynomials $P_i$, $i\in\intent{1}{N}$, are monic and centered, \emph{i.e.} $P_i(z)=z^{d_i}+O\left(z^{d_i-2}\right)$. In fact, replacing $F$ by an affine conjugate, we can always do so~\cite{FriMil}. Note $Q_i(z,w)$, $i\in\intent{0}{N}$ the polynomials in two variables so that
\[F_i\circ\dots\circ F_1(z,w)=\left(Q_i(z,w),Q_{i-1}(z,w)\right),~i\in\intent{1}{N}.\]
It is easy to see that they are defined by the induction relation
\begin{equation}\label{recQN}Q_0=z,\quad Q_1=P_1(z)-a_1w;\qquad Q_i=P_i\circ Q_{i-1}-a_iQ_{i-2},~i\in\intent{2}{N}.\end{equation}
Denote by $V_i(t,w)=t^{D_i}Q_i\left(\frac{1}{t},\frac{w}{t}\right)$. In affine coordinates centered at infinity, the map~$F$ is of the form
\[F(t,w)=\left(\frac{t^{D_N}}{V_N(t,w)},\frac{t^{D_N-D_{N-1}}V_{N-1}(t,w)}{V_N(t,w)}\right).\]

To begin our way towards the normal form, we use B\"ottcher coordinates. Denote by $(t_n,w_n)=F^n(t,w)$ and
\[\phi(t,w)=\left(\varphi_w(t),w\right),\quad \varphi_w(t)=t\prod\limits_{n\in\set{N}}\left(V_N(t_n,w_n)\right)^{-1/D_N^{n+1}}.\]
Here and everywhere after, we consider the principal branch of the $D_N^{n+1}$-th root.

\begin{lem}\label{lemBottcher} The map $\phi$ conjugates $F$ with $F^{(1)}$ (that is $\phi\circ F\circ\phi^{-1}=F^{(1)}$) of the form
  \[F^{(1)}(x,w)=\left(x^{D_N},\frac{A}{D_N}wx^{2D_N-2}(1+\eta(x,w))+h(x)\right),\]
    with $A=\prod_{i=1}^Na_i$, $\eta(x,w)=O\left(x^2,xw\right)$ and
    \[h(x)=x^{D_N}\varphi_0^{-1}(x)^{-D_{N-1}}V_{N-1}\left(\varphi_0^{-1}(x),0\right)\left(1+O\left(x^{D_N+D_{N-1}}\right)\right).\]
  \end{lem}

  \begin{proof} It is well known that the expression of $\varphi_w$ converges (see~\cite{HOV} for standard H\'enon maps), and $t_n=O\left(t^{D_N}\right)$, $w_n=O\left(t^{D_N-D_{N-1}}\right)$, the~$O$ being uniform in~$n$ in a neighbourhood of~$0$. Hence,
    \begin{equation}\label{exprvarphi}\varphi_w(t)=t\left(V_N(t,w)\right)^{-1/D_N}\left(1+O\left(t^{D_N+D_{N-1}}\right)\right).\end{equation}
    Moreover, these coordinates are designed to have $\pi_1\circ\phi\circ F(t,w)=\varphi_w(t)^{D_N}$, $\pi_1$ being the standard projection on the first coordinate. So, the map $F^{(1)}$ is of the form
    \[F^{(1)}(x,w)=\left(x^{D_N},\lambda wx^q(1+\eta(x,w))+h(x)\right),\]
    for some $\lambda\in\set{C}^*$ and $q\in\set{N}$ that we need first to identify. Note that the differentials of the~$F_i$ are given by
    \[DF_i(t,w)=\begin{pmatrix} d_it^{d_i-1}\left(1+O\left(t^2,tw\right)\right) & a_it^{2d_i-1}\left(1+O\left(t^2,tw\right)\right) \\ \left(d_i-1\right)t^{d_i-2}\left(1+O\left(t^2,tw\right)\right) & a_it^{2d_i-2}\left(1+O\left(t^2,tw\right)\right)\end{pmatrix}.\]
    Therefore, using~\eqref{exprvarphi} and the fact that $V_N(t,w)=1+O\left(t^2,tw\right)$,
    \[\begin{aligned}\det\left(D\left(\phi\circ F\circ\phi^{-1}\right)(x,w)\right)&=\prod\limits_{i=1}^N\det\left(DF_i\left(F_{i-1}\circ\dots\circ F_1(x,w)\right)\right)\left(1+O\left(x^2,xw\right)\right)\\
                                                                       &=\prod\limits_{i=1}^Na_i\left(x^{D_{i-1}}\right)^{3d_i-3}\left(1+O\left(x^2,xw\right)\right)\\
                                                                       &=Ax^{3D_N-3}\left(1+O\left(x^2,xw\right)\right).\end{aligned}\]
    Now, since $\det\left(DF^{(1)}(x,w)\right)=\lambda D_Nx^{q+D_N-1}\left(1+\eta(x,w)+w\der{\eta}{w}(x,w)\right)$, we have $\lambda=\frac{A}{D_N}$, $q=2D_N-2$ and $\eta(x,w)=O\left(x^2,xw\right)$. Finally, $h(x)=\pi_2\circ\phi\circ F\circ\phi^{-1}(x,0)$, where $\pi_2$ is the standard projection on the second coordinate. Thus, since~\eqref{exprvarphi} for $t=\varphi_0^{-1}(x)$ and $w=0$ yields $\frac{\varphi_0^{-1}(x)^{D_N}}{V_N\left(\varphi_0^{-1}(x),0\right)}=x^{D_N}\left(1+O\left(x^{D_N+D_{N-1}}\right)\right)$, we obtain the form wanted.                                                       
 \end{proof}

 Next, denote by
 \begin{equation}\label{defphiipsi}\varphi_{0,i}(t)=t\left(V_i(t,0)\right)^{-1/D_i}~\text{and}~\psi_i=\varphi_{0,i}\circ\varphi_{0,i-1}^{-1},\quad i\in\intent{1}{N}.\end{equation}
 By~\eqref{exprvarphi}, we have $\varphi_0(t)=\varphi_{0,N}(t)\left(1+O\left(t^{D_N+D_{N-1}}\right)\right)$ and by Lemma~\ref{lemBottcher},
 \begin{equation}\label{exprhpsi}h(x)=x^{D_N}\left(\psi_N^{-1}(x)\right)^{-D_{N-1}}\left(1+O\left(x^{D_N+D_{N-1}}\right)\right).\end{equation}
 We continue our way towards the normal form.

 \begin{lem}\label{normalformstep2}Keep the notations of Lemma~\ref{lemBottcher}. The germ $F^{(1)}$ is conjugate to the germ
   \[F^{(2)}(x,y)=\left(x^{D_N},\frac{A}{D_N}yx^{2D_N-2}+\wt{h}(x)\right),\]
   with $\wt{h}(x)=h(x)\left(1+O\left(x^{D_N+D_{N-1}}\right)\right)$.
   \end{lem}

   \begin{proof} We use the work of Favre~\cite[pp.~491--494]{Fav}. He shows that we can find a map~$\Psi$ of the form $\Psi(x,w)=(x,w(1+\psi(x,w)))$ such that $\Psi\circ F^{(1)}=F^{(2)}\circ\Psi$. Moreover, he gives an explicit expression for~$\psi$ in the form of a series. Define the operator
     \[T\psi(x,w)=\left(1+\eta(x,w)\right)\psi\circ F^{(1)}(x,w)+h(x)\int_0^1\left(1+\eta_1(x,wt)\right)\der{\psi}{w}\circ F^{(1)}(x,wt)dt,\]
     for $\eta_1=\eta+w\der{\eta}{w}$. Then, $\psi=\sum_{k\in\set{N}}T^k\eta$. Obviously, if $f=O\left(x^2,xw\right)$, then $Tf=O\left(x^2\right)$. Actually, the composition with~$F^{(1)}$ makes degree grow much faster. Therefore, $\psi(x,w)$ is a $O\left(x^2,xw\right)$. Now, this gives
     \[\begin{aligned}\wt{h}(x)&=\pi_2\circ F^{(2)}\circ\Psi(x,0)=\pi_2\circ\Psi\circ F^{(1)}(x,0)=h(x)\left(1+\psi\left(x^{D_N},h(x)\right)\right)\\
         \wt{h}(x)&=h(x)\left(1+O\left(x^{D_N+D_{N-1}}\right)\right).\end{aligned}\]
   \end{proof}

   We finish the computation of the normal form with the third step of Favre's procedure~\cite[pp.~496--498]{Fav}.
   \begin{lem}\label{normalformstep3} Keep the notations of Lemma~\ref{normalformstep2} and write $\wt{h}(x)=\sum_{n\in\set{N}^*}b_nx^n$. Then, $F^{(2)}$ is conjugate with the normal form
     \[F^{(3)}(x,y)=\left(x^{D_N},\frac{A}{D_N}yx^{2D_N-2}+g(x)\right),\]
     where
     \[g(x)=\sum\limits_{n=1}^{2D_N-2}b_nx^n+\frac{D_N}{A}b_{2D_N-1}x^{D_N}+\wt{b}_{2D_N}x^{2D_N},\]
     with $\wt{b}_{2D_N}=0$ if $A\neq D_N$ and $\wt{b}_{2D_N}=b_{2D_N}$ if $A=D_N$.
   \end{lem}

   Now, we will show the following.

   \begin{prop}\label{FmapstoF3inj} The map $F\mapsto F^{(3)}$, obtained by combining Lemmas~\ref{lemBottcher},~\ref{normalformstep2} and~\ref{normalformstep3}, is injective.

   \end{prop}
   
   That is, if two generalized H\'enon maps are given by monic and reduced polynomials as in the beginning of the section, and if the process of the lemmas gives the same normal form, then they are equal. Note that the statement is \emph{not} up to conjugacy (neither for~$F$ nor for~$F^{(3)}$). That's why we have to specify from where we obtain our normal form. For clarity's sake, we begin the proof by some lemmas, where we still work on one map~$F$.

   \begin{lem}\label{lempsii} Recall the maps~$\psi_i$ from~\eqref{defphiipsi} and denote by $U_i(t)=t^{d_i}P_i\left(\frac{1}{t}\right)$. Then,
     \[\psi_i(x)^{D_i}=\frac{x^{D_i}}{U_i\left(x^{D_{i-1}}\right)-a_ix^{D_i}\left(\psi_{i-1}^{-1}(x)\right)^{-D_{i-2}}},\quad i\in\intent{2}{N}.\]
   \end{lem}

   \begin{proof} This comes from the induction relation between the $V_i$. Note that
     \[\psi_i(x)^{D_i}=\frac{\left(\varphi_{0,i-1}^{-1}(x)\right)^{D_i}}{V_i\left(\varphi_{0,i-1}^{-1}(x),0\right)}=\frac{1}{Q_i\left(\frac{1}{\varphi_{0,i-1}^{-1}(x)},0\right)},\]
     by the definition of the $V_i$. Now, using~\eqref{recQN},
     \[\psi_i(x)^{D_i}=\frac{1}{P_i\left(Q_{i-1}\left(\frac{1}{\varphi_{0,i-1}^{-1}(x)},0\right)\right)-a_iQ_{i-2}\left(\frac{1}{\varphi_{0,i-1}^{-1}(x)},0\right)}.\]
     Coming back from~$Q$ to~$V$ , we get
     \[\psi_i(x)^{D_i}=\frac{1}{P_i\left(\left(\varphi_{0,i-1}^{-1}(x)\right)^{-D_{i-1}}V_{i-1}\left(\varphi_{0,i-1}^{-1}(x),0\right)\right)-a_i\left(\varphi_{0,i-1}^{-1}(x)\right)^{-D_{i-2}}V_{i-2}\left(\varphi_{0,i-1}^{-1}(x),0\right)}.\]
     %\[\psi_i(x)^{D_i}=\frac{1}{P_i\left(\frac{V_{i-1}\left(\varphi_{0,i-1}^{-1}(x),0\right)}{\left(\varphi_{0,i-1}^{-1}(x)\right)^{D_{i-1}}}\right)-a_i\frac{V_{i-2}\left(\varphi_{0,i-1}^{-1}(x),0\right)}{\left(\varphi_{0,i-1}^{-1}(x)\right)^{D_{i-2}}}}.\]
     Now, in the denominator we recognize $\varphi_{0,i-1}^{-D_{i-1}}$ and $\varphi_{0,i-2}^{-D_{i-2}}$. Therefore,
     \[\psi_i(x)^{D_i}=\frac{1}{P_i\left(\frac{1}{x^{D_{i-1}}}\right)-a_i\left(\psi_{i-1}^{-1}(x)\right)^{-D_{i-2}}}=\frac{x^{D_i}}{U_i\left(x^{D_{i-1}}\right)-a_ix^{D_i}\left(\psi_{i-1}^{-1}(x)\right)^{-D_{i-2}}}.\qedhere\]
   \end{proof}

   \begin{lem}\label{lempsii2} We have the following approximation of the~$\psi_i$.
     \[\psi_i(x)=x\left(1+\frac{a_i}{D_i}x^{D_i-D_{i-2}}+O\left(x^{2D_{i-1}}\right)\right),\quad i\in\intent{2}{N},\]
       \[\psi_1(x)=x\left(1+O\left(x^2\right)\right).\]
     \end{lem}

     In~$\psi_i$, the term $a_ix^{D_i-D_{i-2}}$ is of course also a $O\left(x^{2D_{i-1}}\right)$ if $d_i\geq3$.

     \begin{proof} Since $\psi_1(x)=x\left(V_1(x,0)\right)^{-1/d_1}=x\left(U_1(x)\right)^{-1/d_1}$, the expansion of $\psi_1$ is just a consequence of~$P_1$ being monic and centered. Let $i$ be greater or equal to~$2$ and suppose that the lemma is proved up to rank $i-1$. By Lemma~\ref{lempsii}, since $\psi_i$ is tangent to the identity,
       \[\psi_i(x)=x\left(U_i\left(x^{D_{i-1}}\right)-a_ix^{D_i}\left(\psi_{i-1}^{-1}(x)\right)^{-D_{i-2}}\right)^{-1/D_i}=x\left(1+\frac{a_i}{D_i}x^{D_i-D_{i-2}}+O\left(x^{2D_{i-1}}\right)\right).\]
       Above, we used that $P_i$ is centered and monic, so that $U_i(x)=1+O\left(x^2\right)$ and
       \[D_i+D_{i-1}-D_{i-2}-D_{i-3}\geq2D_{i-1}.\]
       We conclude by induction.
     \end{proof}

     \begin{proof}[Proof of Proposition~\ref{FmapstoF3inj}] Let~$F$ and~$G$ be two generalized H\'enon maps, defined by compositions of monic and centered standard H\'enon maps
       \[F=F_N\circ\dots\circ F_1,\qquad G=G_N\circ\dots\circ G_1;\]
       \[F_i(z,w)=\left(P_i^{(1)}(z)-a_i^{(1)}w,z\right),\quad G_i(z,w)=\left(P_i^{(2)}(z)-a_i^{(2)}w,z\right),\]
       yielding the same normal form $(x,y)\mapsto\left(x^{D_N},\lambda yx^{2D_N-2}+g(x)\right)$ by Lemmas~\ref{lemBottcher},~\ref{normalformstep2} and~\ref{normalformstep3}. Note that Lemma~\ref{Henonlocegaldeg} implies that indeed $F$ and $G$ are the composition of the same amount of standard H\'enon maps, and that $\deg(F_i)=\deg(G_i)$. Denote by $\varphi_{0,i}^{(j)}$, $\psi_i^{(j)}$, for $j=1$ (respectively~$j=2$) the maps defined in~\eqref{defphiipsi} for~$F$ (respectively for~$G$). Denote by $h^{(1)},\wt{h}^{(1)}$ (respectively $h^{(2)},\wt{h}^{(2)}$) the maps obtained in Lemmas~\ref{lemBottcher} and~\ref{normalformstep2} for~$F$ (respectively~$G$). If $\wt{h}^{(j)}(z)=\sum_{n\in\set{N}}b_n^{(j)}z^n$, for $j\in\{1,2\}$, we have by Lemma~\ref{normalformstep3}
       \[\begin{aligned}g(x)&=\sum\limits_{n=1}^{2D_N-2}b_n^{(1)}x^n+\frac{D_N}{A}b_{2D_N-1}^{(1)}x^{D_N}+\wt{b}^{(1)}_{2D_N}x^{2D_N}\\
                            &=\sum\limits_{n=1}^{2D_N-2}b_n^{(2)}x^n+\frac{D_N}{A}b_{2D_N-1}^{(2)}x^{D_N}+\wt{b}^{(2)}_{2D_N}x^{2D_N}.\end{aligned}\]
  Now, by~\eqref{exprhpsi} and Lemma~\ref{normalformstep2}, $\wt{h}^{(j)}(x)=x^{D_N}\left(\left(\psi_N^{(j)}\right)^{-1}(x)\right)^{-D_{N-1}}+O\left(x^{2D_N}\right)$. Using Lemma~\ref{lempsii2}, we infer that $b_{D_N}^{(j)}=0$. Therefore, if
  \[g(x)=\sum_{n=1}^{2D_N-2}c_nx^n+\wt{c}_{2D_N}x^{2D_N},\]
  $b_n^{(1)}=c_n=b_n^{(2)}$, for $n\in\intent{1}{2D_N-2}$, $n\neq D_N$, and $b_{2D_N-1}^{(1)}=\lambda c_{D_N}=b_{2D_N-1}^{(2)}$. In other words, $\wt{h}^{(1)}(x)=\wt{h}^{(2)}(x)+O\left(x^{2D_N}\right)$.

  Next, recall that $\wt{h}^{(j)}(x)=x^{D_N}\left(\left(\psi_N^{(j)}\right)^{-1}(x)\right)^{-D_{N-1}}+O\left(x^{2D_N}\right)$. It follows that we have $\psi_N^{(1)}(x)=\psi_N^{(2)}(x)\left(1+O\left(x^{D_N+D_{N-1}}\right)\right)$. We will show that if
  \begin{equation}\label{HRpsiiPiai}\psi_i^{(1)}(x)=\psi_i^{(2)}(x)\left(1+O\left(x^{D_i+D_{i-1}}\right)\right),\end{equation}
    for $i\geq2$, then $P_i^{(1)}=P_i^{(2)}$, $a_i^{(1)}=a_i^{(2)}$ and $\psi_{i-1}^{(1)}(x)=\psi_{i-1}^{(2)}(x)\left(1+O\left(x^{D_{i-1}+D_{i-2}}\right)\right)$. This is a consequence of Lemmas~\ref{lempsii} and~\ref{lempsii2}. Indeed, if~\eqref{HRpsiiPiai} holds for some~$i\geq2$, by Lemma~\ref{lempsii},
    \[\frac{x^{D_i}}{\left(\psi_i^{(j)}(x)\right)^{D_i}}=U_i^{(j)}\left(x^{D_{i-1}}\right)-a_i^{(j)}x^{D_i}\left(\left(\psi_{i-1}^{(j)}\right)^{-1}(x)\right)^{-D_{i-2}}.\]
    Therefore,
    \[U_i^{(1)}\left(x^{D_{i-1}}\right)-a_i^{(1)}x^{D_i}\left(\left(\psi_{i-1}^{(1)}\right)^{-1}(x)\right)^{-D_{i-2}}=U_i^{(2)}\left(x^{D_{i-1}}\right)-a_i^{(2)}x^{D_i}\left(\left(\psi_{i-1}^{(2)}\right)^{-1}(x)\right)^{-D_{i-2}},\]
    up to order $D_i+D_{i-1}-1$. Since the $a_i^{(j)}x^{D_i}\left(\left(\psi_{i-1}^{(j)}\right)^{-1}(x)\right)^{-D_{i-2}}$ have no terms multiples of $D_{i-1}$ up to $D_i$ in their expansion by Lemma~\ref{lempsii2}, we obtain $U_i^{(1)}=U_i^{(2)}$. Moreover, $\psi_{i-1}^{(j)}$ are tangent to the identity. Hence, $a_i^{(1)}=a_i^{(2)}$ and $\psi_{i-1}^{(1)}(x)=\psi_{i-1}^{(2)}(x)\left(1+O\left(x^{D_{i-1}+D_{i-2}}\right)\right)$.

    Using the above argument by induction, we obtain $P_i^{(1)}=P_i^{(2)}$, $a_i^{(1)}=a_i^{(2)}$, for each $i\in\intent{2}{N}$, and $\psi_1^{(1)}(x)=\psi_1^{(2)}(x)\left(1+O\left(x^{d_1+1}\right)\right)$. Since $\psi_1^{(j)}(x)=x\left(U_1^{(j)}(x)\right)^{-1/d_1}$, we also get $P_1^{(1)}=P_1^{(2)}$. Finally, we have $A=\prod_{i=1}^N a_i^{(1)}=\prod_{i=1}^Na_i^{(2)}$ because $\lambda=\frac{A}{D_N}$, so that $a_1^{(1)}=a_1^{(2)}$. Hence, $F=G$ and the proposition is proved.
                        
     \end{proof}

     \begin{proof}[End of proof of Theorem~\ref{Henonlocaff}] Again, we use the work of Favre. He shows~\cite[p.~482]{Fav} that two normal forms like in Lemma~\ref{normalformstep3} are conjugate if and only if $g_2(x)=bg_1(\zeta x)$, for some $b\in\set{C}^*$ and $\zeta^{D_N-1}=1$. Identifying the coefficients in $x^{D_N-D_{N-1}}$, we get $b=\zeta^{D_{N-1}-1}$ if both come from the process of the Lemmas, starting with generalized H\'enon maps. It is quite easy to see that this is equivalent to $h_2(x)=\zeta^{D_{N-1}-1}h_1(\zeta x)+O\left(x^{2D_N}\right)$. By Proposition~\ref{FmapstoF3inj}, it is enough to find one generalized H\'enon map~$H$, affinely conjugate to~$F$, that can be written as a composition of monic and centered H\'enon maps, and that yields the normal form with $g_2$. Since~$F$ and~$H$ here will be compositions of standard H\'enon maps with monic and centered polynomials, the conjugacy will even be linear with a very particular shape, but this comes from our first reduction on~$F$ and our choice of~$H$. Consider
       \[H=\theta\circ F\circ\theta^{-1},\quad\theta\colon(z,w)\mapsto\left(\zeta z,\zeta^{D_{N-1}}w\right).\]
       That is,
       \[\begin{aligned}H(z,w)&=\left(\zeta Q_N\left(\zeta^{-1} z,\zeta^{-D_{N-1}}w\right),\zeta^{D_{N-1}}Q_{N-1}\left(\zeta^{-1} z,\zeta^{-D_{N-1}}w\right)\right)\\
           H(z,w)&=\left(R_N(z,w),R_{N-1}(z,w)\right).\end{aligned}\]
       Note that
       \[H=\theta_N\circ F_N\circ\theta_{N-1}^{-1}\circ\theta_{N-1}\circ F_{N-1}\circ\theta_{N-2}^{-1}\circ\dots\circ\theta_1\circ F_1\circ\theta_0^{-1},\]
       with $\theta_i(z,w)=\left(\zeta^{D_i}z,\zeta^{D_{i-1}}w\right)$, and $D_{-1}=D_{N-1}$. In this form, we have
       \[\theta_i\circ F_i\circ\theta_{i-1}^{-1}(z,w)=\left(\zeta^{D_i}P_i\left(\zeta^{-D_{i-1}}z\right)-a_i\zeta^{D_i-D_{i-2}}w,z\right),\]
       which is a standard H\'enon map with monic and centered polynomial. Moreover, if we denote by $W_i(t,w)=t^{D_i}R_i\left(\frac{1}{t},\frac{w}{t}\right)$, it satisfies for $i=N$
       \[W_N(t,0)=t^{D_N}R_N\left(\frac{1}{t},0\right)=t^{D_N}\zeta Q_N\left(\frac{1}{\zeta t},0\right)=V_N\left(\zeta t,0\right).\]
       Similarly, $W_{N-1}(t,0)=V_{N-1}\left(\zeta t,0\right)$. If we denote by $\varphi_{0,i}^{(2)}$, $\psi_i^{(2)}$ the maps defined by~\eqref{defphiipsi} with respect to the $W_i$, we obtain $\varphi_{0,i}^{(2)}(t)=\zeta^{-1}\varphi_{0,i}\left(\zeta t\right)$, for $i\in\{N-1,N\}$. Therefore, $\psi_N^{(2)}(t)=\zeta^{-1}\psi_N\left(\zeta t\right)$. Using~\eqref{exprhpsi}, the map $h^{(2)}$ defined by Lemma~\ref{lemBottcher} applied to~$H$ can be written as
       \[h^{(2)}(x)=x^{D_N}\zeta^{D_{N-1}}\psi_N^{-1}\left(\zeta x\right)^{-D_{N-1}}+O\left(x^{2D_N}\right)=bh_1(\zeta x)+O\left(x^{2D_N}\right)=h_2(x)+O\left(x^{2D_N}\right).\]
       With our first remarks, this concludes the proof.
     \end{proof}

     \section{Surjectivity}\label{secsurj}

     In this section, we compute the image of the map $F\mapsto S_F$ onto Kato surfaces of a given self-intersection profile, with the notations before Theorem~\ref{critKato}. Let us first recall results and notions of~\cite{OelToma} that we will use. Let~$\varphi$ be a Dloussky germ of the form
     \[\varphi(x,y)=\left(x^p,\lambda yx^q+P(x)+cx^{pq/(p-1)}\right),\]
     with~$\lambda\in\set{C}^*$, $p,q\geq2$, $c=0$ if $\lambda\neq1$ or $\frac{q}{p-1}\notin\set{N}$ and $P(x)=x^j+\sum_{l=j+1}^qc_lx^l$, for some $j\in\intent{1}{q}$. The tuple of integers $(p,q,j)$ is called the tuple of \emph{invariants} of~$\varphi$. Define also~\cite[Definition~4.7]{OelToma} the sequences of integers
     \begin{equation}\label{deftype}\begin{aligned}m_1&=j,\quad i_1=\gcd(m_1,p),\\
         m_{\alpha}&=\min\left\{m>m_{\alpha-1}~;~c_m\neq0~\text{and}~\gcd\left(m,i_{\alpha-1}\right)<i_{\alpha-1}\right\},~2\leq\alpha\leq t,\\
         i_{\alpha}&=\gcd\left(i_{\alpha-1},m_{\alpha}\right),~2\leq\alpha\leq t,\\
         i_t&=\gcd\left(m_t,i_{t-1}\right)=\gcd\left(m_1,\dots,m_t,p\right)=1.\end{aligned}\end{equation}
     That is,~$t$ is defined to be the integer when the process gives $i_t=1$. This process indeed ends as such, since~$\varphi$ is a Dloussky germ and by a result of Favre~\cite{Fav}. The sequence $(m_1,\dots,m_t)$ is called the \emph{type} of~$\varphi$. If~$t=1$,~$\varphi$ is called of \emph{simple type}.

     All these integers can in fact be recovered with only the self-intersection profile (in the language of~\cite{OelToma}, the \emph{Dloussky sequence}) of the Kato surface~$S$ induced by~$\varphi$. In fact, according to the Dloussky sequence, one decomposes~$\varphi$ into germs of simple type, then one finds its invariants and type by a reverse Euclid algorithm~\cite[Section~6]{OelToma}, and finally one finds the invariants and type of~$\varphi$ with induction formulas~\cite[Proposition~5.10]{OelToma}.

     Denote by $s_{\ell}=(\ell+2,2,\dots,2)$, with~$2$ repeated $\ell-1$ times and $r_{\ell}=(2,\dots,2)$, with~$2$ repeated~$\ell$ times, both for~$\ell\geq1$. These are called respectively \emph{singular} and \emph{regular} Dloussky sequences and are the elementary bricks for general Dloussky sequences. Note that the index~$\ell$ is the length of the sequences~$s_{\ell}$ and~$r_{\ell}$. Now, we will recover the type of a germ inducing a Kato surface of self-intersection profile~\eqref{DlSHenongen} by applying the process above. We will leave some of the details to the reader. Consider a Kato surface~$S$ of self-intersection profile
     \begin{equation}\label{DlSHenongen2}[s_1,s_{d_1-2},r_{d_1},s_1,s_{d_2-2},r_{d_2},\dots,s_1,s_{d_N-2},r_{d_N}],\end{equation}
     that is of the same (cyclic) profile as in~\eqref{DlSHenongen}. Notice that if for some~$i\in\intent{1}{N}$,~$d_i=2$, the sequence~$s_{d_i-2}$ is empty, but this doesn't change the computation. The decomposition~\cite[p.~335]{OelToma} of a germ~$\varphi$ inducing~$S$ into germs of simple type is then of the form $\varphi=\varphi_N\circ\dots\circ\varphi_1$, with~$\varphi_i$ of Dloussky sequence
     \[[s_1,s_{d_i-2},r_{d_i}],\qquad i\in\intent{1}{N}.\]
     The tuple of invariants of~$\varphi_i$ is then easily computable (see~\cite[pp.~335--337]{OelToma}) to be $(d_i,2d_i-2,d_i-1)$, and its type is then~$(d_i-1)$. Now, obvious inductions using formulas in~\cite[p.~333]{OelToma} give that a Kato surface~$S$ of self-intersection profile~\eqref{DlSHenongen2} is induced by a germ~$\varphi$ of tuple of invariants and type
     \[(p,q,j)=(D_N,2D_N-2,D_N-D_{N-1}),\qquad(m_1,\dots,m_N)=\left(2D_N-D_{N-i+1}-D_{N-i}\right)_{i\in\intent{1}{N}},\]
     with the same notations as in Section~\ref{secnormalform} for the~$D_i$, $i\in\intent{1}{N}$. Coming back to the definition~\eqref{deftype}, we obtain the following.
     \begin{lem}\label{typeKatoHenon} Let $d_1,\dots,d_N\geq2$ be integers, and denote by $D_i=\prod_{j=1}^id_j$, for $i\in\intent{0}{N}$. Let~$S$ be a Kato surface of self-intersection profile
       \[[s_1,s_{d_1-2},r_{d_1},s_1,s_{d_2-2},r_{d_2},\dots,s_1,s_{d_N-2},r_{d_N}].\]
       Then, there exists a Dloussky germ $\varphi\colon\germ{\set{C}^2}{0}\to\germ{\set{C}^2}{0}$, inducing~$S$, of the form
       \begin{equation}\label{formtypeHenon}\begin{aligned}\varphi(x,y)=&\left(x^{D_N},\lambda yx^{2D_N-2}+P(x)+cx^{2D_N}\right),\\
           P(x)=&\sum\limits_{i=1}^Nx^{2D_N-D_{N-i+1}-D_{N-i}}\sum\limits_{l=0}^{d_{N-i+1}-1}\alpha_l^{(i)}x^{lD_{N-i}},\end{aligned}\end{equation}
       with $\lambda\in\set{C}^*$, $\alpha_0^{(i)}\in\set{C}^*$, $i\in\intent{1}{N}$, $\alpha_0^{(1)}=1$ and $c=0$ if $\lambda\neq1$.
       \end{lem}

     Now, let us state the surjectivity property that we will prove.
     \begin{thm}\label{surjHenonKato}Let $d_1,\dots,d_N\geq2$ and $D_i=\prod_{j=1}^id_j$, for $i\in\intent{0}{N}$. Let $\lambda\in\set{C}^*$ and
       \[P(x)=\sum\limits_{i=1}^Nx^{2D_N-D_{N-i+1}-D_{N-i}}\sum\limits_{l=0}^{d_{N-i+1}-1}\alpha_l^{(i)}x^{lD_{N-i}}\]
       be a polynomial with $\alpha_0^{(i)}\in\set{C}^*$, $i\in\intent{1}{N}$ and $\alpha_0^{(1)}=1$. Then, there exists~$c\in\set{C}$, with $c=0$ if $\lambda\neq1$ and a generalized H\'enon map~$F$ such that the germ of~$F$ at infinity is conjugated to
       \[\varphi(x,y)=\left(x^{D_N},\lambda yx^{2D_N-2}+P(x)+cx^{2D_N}\right).\]
     \end{thm}

     Recall that in~\cite{OelToma}, the authors show that the numbers $(\lambda,(\alpha_l^{(i)})_{l,i})$, with $\lambda\neq1$, para\-metrize the logarithmic moduli space of Kato surfaces of self-intersection profile~\eqref{DlSHenongen2} and $((\alpha_l^{(i)})_{l,i},c)$ parametrize the space if $\lambda=1$. In other words, if~$\lambda\neq1$, the map taking a generalized Hénon map to its associated Kato surface is indeed surjective onto the logarithmic moduli space of Kato surfaces with self-intersection profile~\eqref{DlSHenongen2}. On the other hand, when~$\lambda=1$, we get a section of the projection map $((\alpha_l^{(i)})_{l,i},c)\mapsto(\alpha_l^{(i)})_{l,i}$.
     
     Let us now turn to the proof of Proposition~\ref{surjHenonKato} and first slightly change~$\varphi$ in~\eqref{formtypeHenon} and reorder its coefficients. By a change of variable~$(x,y)\mapsto(x,y+\alpha_1^{(1)}x)$, one can conjugate~$\varphi$ with a germ of the form
     \begin{equation}\label{formtypeHenonmod}\begin{aligned}\wt{\varphi}(x,y)=&\left(x^{D_N},\lambda yx^{2D_N-2}+\wt{P}(x)+cx^{2D_N}\right),\\
           \wt{P}(x)=&x^{D_N-D_{N-1}}\sum\limits_{i=1}^Nx^{D_N+D_{N-1}-D_i-D_{i-1}}\left(\wt{\alpha}_0^{(i)}+\sum\limits_{l=2}^{d_i}\wt{\alpha}_l^{(i)}x^{lD_{i-1}}\right),\end{aligned}\end{equation}
     with $\wt{\alpha}_0^{(i)}\in\set{C}^*$ and $\wt{\alpha}_0^{(N)}=1$. In fact, the relationship between the~$\alpha_l^{(i)}$ in~\eqref{formtypeHenon} and the~$\wt{\alpha}_l^{(i)}$ in~\eqref{formtypeHenonmod} is the following
     \[\wt{\alpha}_l^{(i)}=\alpha_l^{(N-i+1)},~l\neq1,d_i,\qquad \wt{\alpha}_{d_i}^{(i)}=\alpha_1^{(N-i+2)},~i\neq1,\qquad\wt{\alpha}_{d_1}^{(1)}=\lambda\alpha_1^{(1)}.\]
     So, now, by Lemmas~\ref{normalformstep2} and~\ref{normalformstep3} and by~\eqref{exprhpsi}, we only need to show the following.

     \begin{prop}\label{propsurjKato} Let~$\wt{P}$ be a polynomial as in~\eqref{formtypeHenonmod} and $\lambda\in\set{C}^*$. Then, there exists a generalized H\'enon map~$F$, with Jacobian~$A$ and degree~$D_N$ with $\lambda=\frac{A}{D_N}$, such that
       \[x^{D_N}\left(\psi_N^{-1}(x)\right)^{-D_{N-1}}=\wt{P}(x)\left(1+O\left(x^{D_N+D_{N-1}}\right)\right),\]
       with the notations of~\eqref{defphiipsi} related to~$F$.
     \end{prop}

     We will use extensively the following elementary result.

     \begin{lem}\label{lempolyuniv} Let $\psi(x)=1+\sum_{l\in\set{N}^*}\beta_lx^l$ be a power series and $\alpha\in\set{R}^*$. Then, developping in power series $\psi(x)^{\alpha}=1+\sum_{l\in\set{N}^*}\gamma_lx^l$, we have
       \[\gamma_l=\alpha\beta_l+P_{l,\alpha}\left(\beta_1,\dots,\beta_{l-1}\right),\qquad l\in\set{N}^*,\]
       where $P_{l,\alpha}$ are universal polynomials. Here, we took the principal branch of the logarithm to determine $\psi(x)^{\alpha}$. Moreover, if $\gcd\left\{m\in\intent{1}{l-1}~;~\beta_m\neq0\right\}$ does not divide~$l$, then $P_{l,\alpha}(\beta_1,\dots,\beta_{l-1})=0$.

       Similarly, if $\psi(x)=x\left(1+\sum_{l\in\set{N}^*}\beta_lx^l\right)$, then $\psi^{-1}(x)=x\left(1+\sum_{l\in\set{N}^*}\gamma_lx^l\right)$, for
       \[\gamma_l=-\beta_l+R_l\left(\beta_1,\dots,\beta_{l-1}\right),\qquad l\in\set{N}^*,\]
       where $R_l$ are universal polynomials. Moreover, if $\gcd\left\{m\in\intent{1}{l-1}~;~\beta_m\neq0\right\}$ does not divide~$l$, then $R_l(\beta_1,\dots,\beta_{l-1})=0$.
     \end{lem}

     \begin{proof}[Proof of Proposition~\ref{propsurjKato}] Consider a generalized H\'enon map $F=F_N\circ\dots\circ F_1$, where
       \[F_i(z,w)=\left(P_i(z)-a_iw,z\right),\qquad i\in\intent{1}{N},\]
       with $a_i\in\set{C}^*$ and~$P_i$ are monic and centered polynomials of degree~$d_i$. Denote by
       \begin{equation}\label{defbetali}U_i(x)=x^{d_i}P_i\left(\frac{1}{x}\right)=1+\sum_{l=2}^{d_i}\beta_l^{(i)}x^i.\end{equation}
       Also, define $\beta_0^{(i)}=a_{i+1}$, for $i\in\intent{1}{N-1}$, $\beta_0^{(N)}=1$ and suppose that $a_1=\frac{\lambda D_N}{\prod_{i=2}^{N-1}a_i}$. We will show that the map $h(x)=x^{D_N}\left(\psi_N^{-1}(x)\right)^{-D_{N-1}}$ equals
         \begin{equation}\label{surjpoly}x^{D_N-D_{N-1}}\sum\limits_{i=1}^Nx^{D_N+D_{N-1}-D_i-D_{i-1}}\left(\gamma_0^{(i)}+\sum\limits_{l=2}^{d_i}\gamma_l^{(i)}x^{lD_{i-1}}\right)\left(1+O\left(x^{D_N+D_{N-1}}\right)\right),\end{equation}
         with for $i\in\intent{1}{N}$ and $l\in\intent{0}{d_i}$, $l\neq1$,
         \[\gamma_l^{(i)}=u_{l,i}\beta_l^{(i)}+S_{l,i}\left(\beta_0^{(N)},\beta_2^{(N)},\dots,\beta_{d_N}^{(N)},\dots,\beta_0^{(i+1)},\beta_2^{(i+1)},\dots,\beta_{d_{i+1}}^{(i+1)},\beta_0^{(i)},\beta_2^{(i)},\dots,\beta_{l-1}^{(i)}\right).\]
         Above, $S_{l,i}$ are polynomials depending only on the degrees $d_1,\dots,d_N$, $u_{l,i}\in\set{C}^*$, $u_{0,N}=1$ and $\gamma_0^{(i)}=u_{0,i}\beta_0^{(i)}$, $i\in\intent{1}{N}$ (\emph{i.e.} the $S_{l,i}$ vanishes in the above equation). If we show the latter, it is clear that we can invert such a map $(\beta_l^{(i)})_{l,i}\mapsto(\gamma_l^{(i)})_{l,i}$ by induction, and then find~$\beta_l^{(i)}$ such that $\gamma_l^{(i)}=\wt{\alpha}_l^{(i)}$. The conditions for $l=0$ ensure that we indeed obtain $a_i\neq0$, the other coefficients of the H\'enon map in~\eqref{defbetali} being free to be any complex number.

         Now, to show~\eqref{surjpoly}, we proceed by induction on~$N$. Since $\psi_1(x)=xU_1(x)^{-1/d_1}$, Lemma~\ref{lempolyuniv} gives the result for~$N=1$. For the induction, we use Lemma~\ref{lempsii} and repeatedly Lemma~\ref{lempolyuniv}. The details are left to the reader.
         %use repeatedly Lemmas~\ref{lempolyuniv},~\ref{lempsii} and~\ref{lempsii2} by induction on~$N$, with arguments similar to the proof of Proposition~\ref{FmapstoF3inj}. The details are left to the reader.
     \end{proof}

\bibliography{Henon-Kato}

\bibliographystyle{plain}

\end{document}

%% file: Construction-Kato.tex
\def \globalscale {9.600000}
\begin{tikzpicture}[y=0.80pt, x=0.80pt, yscale=-\globalscale, xscale=\globalscale, inner sep=0pt, outer sep=0pt,line width=0.5pt,draw=black]
\draw[fill=black,fill opacity=0.2]  (18.1240,36.1156) node[above,yshift=3.55cm,opacity=1] {$\Pi^{-1}\left(\Sigma_{\eps}\right)$}circle
  (0.4208cm);

\draw (47.3549,37.5232) circle
(0.1266cm);

\draw[fill=black] (47.3549,37.5232) node[below,yshift=-0.1cm] {$0$}circle (0.004cm);

\draw[->,>=stealth]
  (51.9626,38.6933) .. controls (52.1705,38.9220) and (52.4451,39.0894) ..
  (52.7437,39.1695) .. controls (53.0423,39.2495) and (53.3638,39.2418) ..
  (53.6583,39.1476) .. controls (54.1183,39.0006) and (54.5040,38.6371) ..
  (54.6781,38.1867) node[above right,xshift=0.1cm] {$F=\Pi\circ\sigma$} .. controls (54.8523,37.7362) and (54.8113,37.2079) ..
  (54.5698,36.7896) .. controls (54.3283,36.3713) and (53.8913,36.0717) ..
  (53.4141,35.9973) .. controls (52.9369,35.9228) and (52.4293,36.0751) ..
  (52.0719,36.4000);

\draw[fill=white] (18.1240,36.1156) circle
  (0.3685cm);

\draw
  (10.6181,28.5750) node[above] {$D_1$} -- (10.6181,41.8042);

\draw
  (9.1273,30.2077) node[left] {$D_2$} -- (22.7380,30.2077);

\draw
  (21.4312,28.5750) node[above] {$D_3$} -- (21.4312,41.8042);

\draw[fill=black,fill opacity=0.2] (21.4312,36.2479) node[below,yshift=-0.285cm,xshift=0.6cm,opacity=1] {$\sigma\left(\Sigma_{\eps}\right)$} circle
  (0.0735cm);

\draw[fill=black,fill opacity=0.7] (21.4312,36.2479) circle
(0.0438cm);

\draw[fill=black] (21.4312,36.2479) circle (0.004cm);

\draw[->,>=stealth]
  (10.6067,46.7546) .. controls (11.2570,44.7416) and (12.3050,42.8579) ..
  (13.6726,41.2438) node[below right] {$\sigma\circ\Pi$} .. controls (15.1114,39.5457) and (16.9031,38.1480) ..
  (18.9003,37.1657);

\draw[->,>=stealth]
  (43.9460,34.5747) .. controls (42.0189,32.3103) and (39.3217,30.7135) ..
  (36.4096,30.1129) node[above left,yshift=0.15cm,xshift=-0.2cm] {$\sigma$} .. controls (34.0762,29.6316) and (31.6152,29.7842) ..
  (29.3591,30.5501) .. controls (27.1031,31.3160) and (25.0578,32.6931) ..
  (23.4996,34.4955);

\draw[->,>=stealth]
  (32.7010,39.4447) .. controls (34.1360,40.2012) and (35.7515,40.6125) ..
  (37.3736,40.6344) node[below,yshift=-0.1cm] {$\Pi$} .. controls (39.3465,40.6611) and (41.3239,40.1068) ..
  (42.9955,39.0586);

\end{tikzpicture}

%% file: Henon-Kato_v6_arxiv.bbl
\begin{thebibliography}{10}

\bibitem{Kodclass}
Wolf~P. Barth, Klaus Hulek, Chris A.~M. Peters, and Antonius Van~de Ven.
\newblock {\em Compact complex surfaces}, volume~4 of {\em Ergebnisse der
  Mathematik und ihrer Grenzgebiete. 3. Folge. A Series of Modern Surveys in
  Mathematics [Results in Mathematics and Related Areas. 3rd Series. A Series
  of Modern Surveys in Mathematics]}.
\newblock Springer-Verlag, Berlin, second edition, 2004.

\bibitem{BedSmiIV}
Eric Bedford, Mikhail Lyubich, and John Smillie.
\newblock Polynomial diffeomorphisms of {${\bf C}^2$}. {IV}. {T}he measure of
  maximal entropy and laminar currents.
\newblock {\em Invent. Math.}, 112(1):77--125, 1993.

\bibitem{BedSmiI}
Eric Bedford and John Smillie.
\newblock Polynomial diffeomorphisms of {${\bf C}^2$}: currents, equilibrium
  measure and hyperbolicity.
\newblock {\em Invent. Math.}, 103(1):69--99, 1991.

\bibitem{BedSmiIII}
Eric Bedford and John Smillie.
\newblock Polynomial diffeomorphisms of {$\bold C^2$}. {III}. {E}rgodicity,
  exponents and entropy of the equilibrium measure.
\newblock {\em Math. Ann.}, 294(3):395--420, 1992.

\bibitem{BedSmiV}
Eric Bedford and John Smillie.
\newblock Polynomial diffeomorphisms of {${\bf C}^2$}. {V}. {C}ritical points
  and {L}yapunov exponents.
\newblock {\em J. Geom. Anal.}, 8(3):349--383, 1998.

\bibitem{BedSmiVI}
Eric Bedford and John Smillie.
\newblock Polynomial diffeomorphisms of {${\bf C}^2$}. {VI}. {C}onnectivity of
  {$J$}.
\newblock {\em Ann. of Math. (2)}, 148(2):695--735, 1998.

\bibitem{BeneCarl}
Michael Benedicks and Lennart Carleson.
\newblock The dynamics of the {H}\'enon map.
\newblock {\em Ann. of Math. (2)}, 133(1):73--169, 1991.

\bibitem{BRT}
Sylvain Bonnot, Remus Radu, and Raluca Tanase.
\newblock H\'enon mappings with biholomorphic escaping sets.
\newblock {\em Complex Anal. Synerg.}, 3(1):Paper No. 3, 18, 2017.

\bibitem{Xen}
Julia~X\'{e}nelkis de~H\'{e}non.
\newblock H\'enon maps: A {L}ist of {O}pen {P}roblems.
\newblock {\em Arnold Math J.}, 2024.

\bibitem{DlOefol}
G.~Dloussky and K.~Oeljeklaus.
\newblock Vector fields and foliations on compact surfaces of class {$\rm
  VII_0$}.
\newblock {\em Ann. Inst. Fourier (Grenoble)}, 49(5):1503--1545, 1999.

\bibitem{memDlou}
Georges Dloussky.
\newblock Structure des surfaces de {K}ato.
\newblock {\em M\'em. Soc. Math. France (N.S.)}, (14):ii+120, 1984.

\bibitem{Dlou2}
Georges Dloussky.
\newblock Classification of germs of contracting holomorphic mappings.
\newblock {\em Math. Ann.}, 280(4):649--661, 1988.

\bibitem{Dlou1}
Georges Dloussky.
\newblock An elementary construction of {I}noue--{H}irzebruch surfaces.
\newblock {\em Math. Ann.}, 280(4):663--682, 1988.

\bibitem{Dlou4}
Georges Dloussky.
\newblock Non {K}\"ahlerian surfaces with a cycle of rational curves.
\newblock {\em Complex Manifolds}, 8(1):208--222, 2021.

\bibitem{HenonDlou}
Georges Dloussky and Karl Oeljeklaus.
\newblock Surfaces de la classe {VII{$_0$}} et automorphismes de {H}\'enon.
\newblock {\em C. R. Acad. Sci. Paris S\'er. I Math.}, 328(7):609--612, 1999.

\bibitem{Dlou3}
Georges Dloussky, Karl Oeljeklaus, and Matei Toma.
\newblock Class {$\rm VII_0$} surfaces with {$b_2$} curves.
\newblock {\em Tohoku Math. J. (2)}, 55(2):283--309, 2003.

\bibitem{Enoki1}
Ichiro Enoki.
\newblock Surfaces of class {${\rm VII}\sb{0}$}\ with curves.
\newblock {\em Tohoku Math. J. (2)}, 33(4):453--492, 1981.

\bibitem{Enoki2}
Ichiro Enoki.
\newblock Deformations of surfaces containing global spherical shells.
\newblock In {\em Classification of algebraic and analytic manifolds ({K}atata,
  1982)}, volume~39 of {\em Progr. Math.}, pages 45--64. Birkh\"auser Boston,
  Boston, MA, 1983.

\bibitem{Fav}
Charles Favre.
\newblock Classification of 2-dimensional contracting rigid germs and {K}ato
  surfaces. {I}.
\newblock {\em J. Math. Pures Appl. (9)}, 79(5):475--514, 2000.

\bibitem{theseFav}
Charles Favre.
\newblock {\em {Dynamique des applications rationnelles}}.
\newblock Th\`{e}ses, {Universit{\'e} Paris Sud - Paris XI}, Jan 2000.

\bibitem{ForSibHenon}
John~Erik Forn\ae{}ss and Nessim Sibony.
\newblock Complex {H}\'enon mappings in {${\bf C}^2$} and {F}atou-{B}ieberbach
  domains.
\newblock {\em Duke Math. J.}, 65(2):345--380, 1992.

\bibitem{FriMil}
Shmuel Friedland and John Milnor.
\newblock Dynamical properties of plane polynomial automorphisms.
\newblock {\em Ergodic Theory Dynam. Systems}, 9(1):67--99, 1989.

\bibitem{Henon}
Michel H\'enon.
\newblock A two-dimensional mapping with a strange attractor.
\newblock {\em Comm. Math. Phys.}, 50(1):69--77, 1976.

\bibitem{HOV}
John~H. Hubbard and Ralph~W. Oberste-Vorth.
\newblock H\'enon mappings in the complex domain. {I}. {T}he global topology of
  dynamical space.
\newblock {\em Inst. Hautes \'Etudes Sci. Publ. Math.}, (79):5--46, 1994.

\bibitem{Jakob}
Michael Jakobson.
\newblock Parameter choice for families of maps with many critical points.
\newblock In {\em Modern dynamical systems and applications}, pages 359--364.
  Cambridge Univ. Press, Cambridge, 2004.

\bibitem{Jung}
Heinrich W.~E. Jung.
\newblock \"uber ganze birationale {T}ransformationen der {E}bene.
\newblock {\em J. Reine Angew. Math.}, 184:161--174, 1942.

\bibitem{Kato}
Masahide Kato.
\newblock Compact complex manifolds containing ``global''\ spherical shells.
  {I}.
\newblock In {\em Proceedings of the {I}nternational {S}ymposium on {A}lgebraic
  {G}eometry ({K}yoto {U}niv., {K}yoto, 1977)}, pages 45--84. Kinokuniya Book
  Store, Tokyo, 1978.

\bibitem{Kod66}
Kunihiko Kodaira.
\newblock On the structure of compact complex analytic surfaces. {I}.
\newblock {\em Amer. J. Math.}, 86:751--798, 1964.

\bibitem{Nak1}
Iku Nakamura.
\newblock On surfaces of class {${\rm VII}_0$} with curves.
\newblock {\em Invent. Math.}, 78(3):393--443, 1984.

\bibitem{Nak2}
Iku Nakamura.
\newblock On surfaces of class {${\rm VII}_0$} with curves. {II}.
\newblock {\em Tohoku Math. J. (2)}, 42(4):475--516, 1990.

\bibitem{OelToma}
Karl Oeljeklaus and Matei Toma.
\newblock Logarithmic moduli spaces for surfaces of class {VII}.
\newblock {\em Math. Ann.}, 341(2):323--345, 2008.

\bibitem{Pal}
Ratna Pal.
\newblock Relation between {H}\'enon maps with biholomorphic escaping sets.
\newblock {\em Math. Ann.}, 388(4):4355--4382, 2024.

\bibitem{SibPanSyn}
Nessim Sibony.
\newblock Dynamique des applications rationnelles de {$\bold P^k$}.
\newblock In {\em Dynamique et g\'eom\'etrie complexes ({L}yon, 1997)},
  volume~8 of {\em Panor. Synth\`eses}, pages ix--x, xi--xii, 97--185. Soc.
  Math. France, Paris, 1999.

\end{thebibliography}
